\documentclass[11pt, reqno]{amsart}
\usepackage{amsmath}
\usepackage{amssymb}
\usepackage{amsthm}
\usepackage{enumerate}
\usepackage{xcolor}
\usepackage{mathrsfs}
\usepackage{hyperref}
\usepackage[abbrev]{amsrefs}
\usepackage[utf8]{inputenc}
\usepackage{ifthen}
\usepackage{enumitem}
\usepackage{bm}
\usepackage{graphicx}
\usepackage[all]{xy}
\newtheorem{thm}{}[section]
\newtheorem{theorem}[thm]{Theorem}
\newtheorem{corollary}[thm]{Corollary}
\newtheorem{lemma}[thm]{Lemma}

\theoremstyle{definition}

\theoremstyle{remark}
\newtheorem{remark}[thm]{Remark}

\numberwithin{equation}{section}
\allowdisplaybreaks

\newcommand{\N}{\mathbb{N}}

\newcommand{\C}{\mathbb{C}}
\newcommand{\T}{\mathbb{T}}

\newcommand{\abs}[1]{\left\lvert #1\right\rvert}
\newcommand{\norm}[1]{\left\lVert #1\right\rVert}
\renewcommand{\epsilon}{\varepsilon}
\newcommand{\M}{M}
\newcommand{\DD}{\mathbb{D}}

\AtBeginDocument{
\def\MR#1{}
}

\begin{document}
\title[Mean convergence of Lagrange interpolation]{Mean convergence of Lagrange interpolation}
\author[G. Bello]{Glenier Bello}
\address{Departamento de Matem\'{a}ticas\\
Universidad Aut\'{o}noma de Madrid\\
28049 Madrid\\
Spain; 
ICMAT CSIC-UAM-UC3M-UCM\\
28049 Madrid, 
Spain}
\email{glenier.bello@uam.es}
\author[M. Bello-Hern\'{andez}]{Manuel Bello-Hern\'{andez}}
\address{Departamento de Matem\'{a}ticas y Computaci\'{o}n\\
Universidad de La Rioja\\
26004 Logro\~no\\
Spain}
\email{mbello@unirioja.es}
\begin{abstract}
In this note we prove mean convergence of Lagrange interpolation at the zeros of para-orthogonal polynomials for measures in the unit circle which does not belong to Szeg\H{o}'s class in the unit circle. When the measure is in  Szeg\H{o}'s class mean convergence of Lagrange interpolation is proved for functions in the disk algebra.
\end{abstract}
\maketitle
\section{Introduction}
Mean convergence of Lagrange interpolation on zeros of orthogonal polynomials in the real line 
is a topic which has been extensively studied. 
Just to name a few works, let us recall the foundational paper of Erd\"{o}s and Tur\'{a}n \cite{ErdTur37}, 
the trilogy \cite{Nev76,Nev80,Nev1984} of Nevai, and also \cite{Ask72,MasMil08,VerXu05}. 
Marcinkiewicz \cite{Mar36} showed that there is a strong connection between mean convergence of Lagrange interpolation and mean convergence of Fourier series.

Interpolation with nodes at the unit circle has been studied for functions in the disk algebra $A(\overline\DD)$ 
(i.e., functions analytic in the unit disk $\DD$ which are continuous in $\overline{\DD}$), mostly using unit roots or perturbations of these points (cf.\ \cite{ChuiSheZho93,Lub00,Lub02,Wal65}). 
One of the reasons to consider functions in $A(\overline\DD)$ is that, given a nontrivial finite positive Borel measure $\mu$ on the unit circle $\T$, the algebraic polynomials are dense in $L^2(\mu)$ if and only if
$
\log\mu'\not\in L^1(\mathfrak{m}),
$
where  $\mathfrak{m}$ is the Lebesgue measure in $\T$ and $\mu'$ is the Radon--Nikodym derivative of $\mu$ with respect to $\mathfrak{m}$. This is due to Szeg\H{o}--Kolmogorov--Krein's theorem (see for example \cite[Chapter~1]{Ger61} or \cite[Chapter~3]{GreSze84}). 

In this note we are concerned with the mean convergence of Lagrange interpolation at the zeros of para-orthogonal polynomials for measures in the unit circle. The zeros of para-ortogonal polynomials are the nodes of the Szeg\H{o}'s quadrature formula, which is the analogous in the unit circle to Gaussian quadrature formula 
(cf.\ \cite{CanMorVel02,RuyPabPer09,Ber15,Gol02,JonNjaThr89}). 
These zeros share several properties with zeros of orthogonal polynomials with respect to a measure on the real line.

Consider a nontrivial finite positive Borel measure $\mu$ on the unit circle, 
and let $\{w_n\}_{n=0}^\infty$ be a sequence of points in $\T$. 
For each $n$, let $Z_n(w_n)$ denote the $n+1$ zeros of para-orthogonal polynomial of degree $n+1$ associated to $w_n$ and $\mu$. 
It is well known that $Z_n(w_n)\subset\T$. 
Given a function $f$ defined in $\T$, 
let $L_n(f)$ denote the $n$-th Lagrange interpolation polynomial of $f$ with respect to $Z_n(w_n)$. 
We obtain mean convergence of $L_n(f)$ in two different settings.

\begin{theorem}\label{teoMeanConvIntL2} 
Suppose that $\log\mu'\not\in L^1(\mathfrak{m})$. 
If $f\in C(\T)$ then for all $p\in(0,2]$ we have 
\[
\lim_{n\to\infty}\norm{f-L_n(f)}_p=0.
\]
\end{theorem}

An analogous result holds for all nontrivial finite positive measures if we consider 
Lagrange interpolation of functions in the disk algebra $A(\overline\DD)$. 

\begin{theorem}\label{teoDiskAlg} 
If $f\in A(\overline\DD)$ then for all $p\in(0,2]$ we have 
\[
\lim_{n\to\infty}\norm{f-L_n(f)}_p=0.
\]
\end{theorem}

Recall that mean convergence implies the existence of a subsequence of $\{L_n(f)\}$ which converges almost everywhere to $f$. This contrasts with \cite{Ver82}, where V\'{e}rtesi proved that for any infinite triangular array in the unit circle, there exists a function in $A(\overline\DD)$ such that the sequence of its Lagrange interpolants diverge almost everywhere in the unit circle. 

Our results delve into the close relationship between the zeros of orthogonal polynomials in the real line and the zeros of para-orthogonal polynomials.

The paper is organized as follows. 
In Section~\ref{SecLag} we recall some definitions and well known facts about Lagrange interpolation. 
In Section~\ref{sec:proofs} we prove both theorems above.

\section{Lagrange interpolation}\label{SecLag}
Let $\mu$ be a finite positive Borel measure on the unit circle $\T$ which is nontrivial i.e. its support has infinite points, 
and consider the Hilbert space $L^2(\mu)$ of $\mu$-square-integrable functions with inner product 
\[
\left\langle f,g\right\rangle:=\int_\T f\bar g\, d\mu,
\]
and associated norm\footnote{For $0<p<1$ we also use the same notation.} when $1\le p<\infty$,
\[
\|f\|_p=\left(\int |f|^p\,d\mu\right)^{1/p}.
\]
Let $\{\varphi_n\}_{n=0}^\infty$ be the sequence of orthonormal polynomials with respect to $\mu$ 
with positive leading coefficients; that is, these are polynomials of the form 
\[
\varphi_n(z)=\kappa_nz^n+\text{lower degree terms},\quad \kappa_n>0,
\]
satisfying the orthonormality relations $\langle\varphi_n,\varphi_m\rangle=\delta_{nm}$ for all $n,m\ge0$. 

We denote by $\Pi_n$ the set of polynomials of degree less than or equal to $n$. 
Given $p\in\Pi_n$, let $p^*(z):=z^n\overline{p(1/\bar z)}$ be the reversed polynomial of $p$.

The Christoffel kernel is defined by
$$
K_n(w,z)=\sum_{j=0}^n\overline{\varphi_j(w)}\varphi_j(z).
$$
The reproducing kernel property of $K_n$ says that for every $p\in\Pi_n$ we have 
\begin{equation}\label{eq:kernel}
\left\langle p,K_n(w,\cdot)\right\rangle=\int p(z)\overline{K_n(w,z)}\, d\mu(z)=p(w). 
\end{equation}
By the Christoffel--Darboux formula (cf.\ \cite[Theorem~11.4.2]{Sze75}), 
\begin{equation}\label{eq:ChrDar}
K_n(w,z)=\frac{\overline{\varphi_{n+1}^*(w)}\varphi_{n+1}^*(z)-\overline{\varphi_{n+1}(w)}\varphi_{n+1}(z)}{1-\overline{w}z},
\end{equation}
for all $w\neq z$. 
Para-orthogonal polynomials are defined as 
\[
B_{n+1}(w,z):=(1-\overline{w}z)K_n(w,z)
=\overline{\varphi_{n+1}^*(w)}\varphi_{n+1}^*(z)-\overline{\varphi_{n+1}(w)}\varphi_{n+1}(z).
\]
These polynomials were introduced in \cite[Section~6]{JonNjaThr89}. 

For the rest of this section, let us fix $n\in\N$ and $w\in\T$. 
We denote by $Z_n(w)$ the zeros of $B_{n+1}(w,\cdot)$. 
It is well known that these zeros are simple and lie in the unit circle 
(cf.\ \cite{CanMorVel02,Gol02,JonNjaThr89}). 
Hence 
\[
Z_n(w)=\{\zeta_{0n},\zeta_{1n},\ldots,\zeta_{nn}\}\subset\T. 
\]
Clearly, one of the zeros is $w$. 

Next, we recall a symmetric property concerning the zeros of para-orthogonal polynomials. 
We include short proof for an easy reading and to emphasize the connection with the formulas above. 

\begin{lemma}[see \cite{Gol02}]\label{lem:symmetricProp1}
The identity 
\[
K_n(\zeta_{jn},\zeta_{mn})=0
\]
holds for all $j,m\in\{0,1,\dotsc,n\}$ distinct. In particular, for all $j,m\in\{0,1,\dotsc,n\}$ we have 
\[
Z_n(\zeta_{jn})=Z_n(\zeta_{mn}). 
\]
\end{lemma}

\begin{proof}
Take distinct indexes $j,m$. Since $\zeta_{jn},\zeta_{mn}\in Z_n(w)$, we have 
\[
\frac{\varphi_{n+1}(\zeta_{jn})}{\varphi_{n+1}^*(\zeta_{jn})}
=
\frac{\overline{\varphi_{n+1}^*(w)}}{\overline{\varphi_{n+1}(w)}}
=
\frac{\varphi_{n+1}(\zeta_{mn})}{\varphi_{n+1}^*(\zeta_{mn})}
\]
Since these quotients have modulus $1$, we also know that 
\[
\frac{\varphi_{n+1}(\zeta_{jn})}{\varphi_{n+1}^*(\zeta_{jn})}
=
\frac{\overline{\varphi_{n+1}^*(\zeta_{jn})}}{\overline{\varphi_{n+1}(\zeta_{jn})}}.
\]
Hence the statements follow using \eqref{eq:ChrDar}. 
\end{proof}

Now let us consider the polynomials
$$
L_{jn}(z)=\frac{K_n(\zeta_{jn},z)}{\sqrt{K_n(\zeta_{jn},\zeta_{jn})}},\quad j=0,1,\ldots,n. 
$$
The following well known result will be crucial in the proofs of our two theorems in Section~\ref{sec:proofs}. 

\begin{lemma}[see \cite{Gol02}]\label{lem:basis}
The polynomials $\{L_{jn}\}_{j=0}^n$ form an orthonormal basis of $\Pi_n$. 
\end{lemma}

\begin{proof}
Using \eqref{eq:kernel} and Lemma~\ref{lem:symmetricProp1} we have 
\[
\int L_{jn}\overline{L_{mn}}\, d\mu
=
\frac{K_n(\zeta_{jn},\zeta_{mn})}{\sqrt{K_n(\zeta_{jn},\zeta_{jn})}\sqrt{K_n(\zeta_{mn},\zeta_{mn})}}
=
\delta_{jm}.\qedhere
\]
\end{proof}

Let $\mu_n$ be the measure associated to $Z_n(w)$ given by
\begin{equation}
\label{medDiscreta}
d\mu_n=\sum_{j=0}^n\frac{d\delta_{\zeta_{jn}}}{K_n(\zeta_{jn},\zeta_{jn})}. 
\end{equation}

The Szeg\H{o} quadrature (cf.\ \cite{JonNjaThr89}) is defined by
\[
Q_n(f):=\int f\, d\mu_n=\sum_{j=0}^n\frac{f(\zeta_{jn})}{K_n(\zeta_{jn},\zeta_{jn})}.
\]
 
\begin{lemma}[see \cite{Gol02}]\label{lem:KeyLemma}
If $p,q\in\Pi_n$ then 
\[
Q_n(p\bar q)=\int p\bar q\, d\mu_n=\int p\bar q\, d\mu.
\]
\end{lemma}

\begin{proof}
By definition of $\mu_n$ we have 
\[
\int p\bar q\, d\mu_n=\sum_{j=0}^n\frac{p(\zeta_{jn})}{\sqrt{K_n(\zeta_{jn},\zeta_{jn})}}
\frac{\overline{q(\zeta_{jn})}}{\sqrt{K_n(\zeta_{jn},\zeta_{jn})}}.
\]
Using \eqref{eq:kernel} we get 
\[
\int p\overline{L_{jn}}\, d\mu=\frac{p(\zeta_{jn})}{\sqrt{K_n(\zeta_{jn},\zeta_{jn})}}. 
\]
Now the statement follows using Lemma~\ref{lem:basis}. 
\end{proof}

\begin{corollary}[\cite{JonNjaThr89}]\label{cor:LaurentPoly}
If $\Lambda$ is a Laurent polynomial of degree $n$ in $z$ and $1/z$, then 
\[
Q_n(\Lambda)=\int \Lambda\, d\mu_n=\int \Lambda\, d\mu.
\]
\end{corollary}

\begin{proof}
By the fundamental theorem of algebra there exist polynomials $p$ and $q$
in $\Pi_n$ such that $\Lambda(z,1/z)=p(z)\overline{q(z)}$. 
Hence the statement follows directly from Lemma~\ref{lem:KeyLemma}. 
\end{proof}

\begin{corollary}[\cite{JonNjaThr89}] \label{CorConvCuad} If $f\in C(\T)$ then 
\[
\lim_{n\to\infty}Q_n(f)=\int f\, d\mu.
\]
\end{corollary} 

\begin{proof}
This follows immediately from the the density in the uniform norm of Laurent polynomials in $C(\T)$ 
and Corollary~\ref{cor:LaurentPoly}.
\end{proof}

\begin{remark}
Notice that taking $p=q=1$ in Lemma~\ref{lem:KeyLemma} we get 
\begin{equation}\label{eq:M}
\mu_n(\T)=\mu(\T)=:\M<\infty.
\end{equation}
\end{remark}

Let us denote the fundamental polynomials for $Z_n(w)$ by 
$$
\ell_{jn}(z)=\frac{K_n(\zeta_{jn},z)}{K_n(\zeta_{jn},\zeta_{jn})},\quad j=0,1,\ldots,n.
$$
These are polynomials of degree $n$  that satisfy  
\begin{equation}\label{eq:elldelta}
\ell_{jn}(\zeta_{kn})=\delta_{jk}.
\end{equation}
The $n-$th Lagrange interpolation polynomial with respect to $Z_n(w)$ of a function $f$ defined in $\T$ 
is defined as
\[
L_n(f)(z):=\sum_{j=0}^nf(\zeta_{jn})\ell_{jn}(z). 
\]

\begin{remark}\label{rem:ljnOB}
Note that 
\[
\ell_{jn}\sqrt{K_n(\zeta_{jn},\zeta_{jn})}=L_{jn}, \quad j=0,1,\dotsc,n,
\]
so due to Lemma~\ref{lem:basis} the fundamental polynomials form an orthogonal basis of $\Pi_n$. 
Since 
\[
\sum_{j=0}^{n}\ell_{jn}(z)=L_n(1)=1
\]
for all $z\in\T$, we also have 
\[
\sum_{j=0}^{n}\norm{\ell_{jn}}_2^2=\Big\lVert\sum_{j=0}^{n}\ell_{jn}\Big\rVert_2^2
=\mu(\T)=M.
\]
\end{remark}

\begin{lemma} \label{lemsquareF} 
For any function $f\colon\T\to\C$ we have 
\[
\norm{L_n(f)}_2^2=Q_n(\abs f^2).
\]
\end{lemma}

\begin{proof}
Using Lemma~\ref{lem:KeyLemma} we obtain 
\[
\|L_n(f)\|_2^2=\int |L_n(f)|^2\, d\mu=\int |L_n(f)|^2\, d\mu_n=Q_n(|f|^2). 
\qedhere
\]
\end{proof}

\section{Proofs of Theorems~\ref{teoMeanConvIntL2} and \ref{teoDiskAlg}}\label{sec:proofs}
In this section we prove our two main results. 
So let us consider a nontrivial finite positive Borel measure $\mu$ on $\T$, 
and fix a sequence of points $\{w_n\}_{n=0}^{\infty}$ in $\T$. 
As in Section~\ref{SecLag}, 
for each $n$, let $Z_n(w_n)$ denote the zeros of para-orthogonal polynomials associated to $w_n$ and $\mu$, and 
$L_n(f)$ denote the $n$-th Lagrange interpolation polynomial with respect to $Z_n(w_n)$ 
of a function $f$. 

\begin{proof}[Proof of Theorem \ref{teoMeanConvIntL2}] 
It is enough to consider the case $p=2$. Indeed, once this case is proved, the case $0<p<2$ follows using H\"older's inequality. So let us focus on $p=2$.

Let $f\in C(\T)$. Fix $\epsilon>0$.
Since $\log \mu'\not\in L^2(\mathfrak{m})$, by the Szeg\H{o}-Kolmogorov-Krein theorem 
there exists a polynomial $\Pi$ such that
\begin{equation}
\label{aproxL2f}
\|f-\Pi\|_2^2<\epsilon.
\end{equation}
Taking $n$ larger than the degree of $\Pi$, we have
\[
\begin{split}
\norm{f-L_n(f)}_2^2
&\le 
\left(\norm{f-\Pi}_2+\norm{L_n(\Pi-f)}_2\right)^2\\
&\le
2\left(\norm{f-\Pi}_2^2+\norm{L_n(\Pi-f)}_2^2\right).
\end{split}
\]
Thus, according to \eqref{aproxL2f}, Lemma \ref{lemsquareF}, and Corollary \ref{CorConvCuad}, we get 
\[
\limsup_{n\to\infty}\|f-L_n(f)\|_2^2\le 4\epsilon. 
\]
Therefore $\norm{f-L_n(f)}_2$ converges to $0$, as we wanted to prove. 
\end{proof}

As an immediate consequence of the proof of Theorem~\ref{teoMeanConvIntL2}, we obtain the following result.

\begin{corollary} We have
\begin{equation}\label{convInt}
\lim_{n\to\infty}\|f-L_n(f)\|_2=0, \quad \forall f\in C(\T),
\end{equation}
if and only if polynomials are dense in $L^2(\mu)$.

\end{corollary}

\begin{proof} By the proof of Theorem~\ref{teoMeanConvIntL2} we have that if polynomials are dense in $L^2(\mu)$ then \eqref{convInt} holds. The other implication is obvius since the continuous functions are dense in $L^2(\mu)$.
\end{proof}


\begin{proof}[Proof of Theorem \ref{teoDiskAlg}]
As in the proof of Theorem~\ref{teoMeanConvIntL2}, we just need to focus on the case $p=2$. 

It is well known that the functions in $A(\overline\DD)$ are uniform limit of algebraic polynomials in $\overline{\DD}$. 
Thus, given $f\in A(\overline\DD)$, there exits a sequence of algebraic polynomials $\{P_n\}$ such that $\deg(P_n)\le n$ and
\[
\norm{f-P_n}_\infty=\max_{z\in\T}\abs{f(z)-P_n(z)} \underset{n\to\infty}{\longrightarrow}0.
\]
Clearly, this implies 
\[
\norm{f-P_n}_2\underset{n\to\infty}{\longrightarrow}0. 
\]
Now, using that $L_n$ is a projection on $\Pi_n$ and Remark~\ref{rem:ljnOB} we have 
\[
\begin{split}
\norm{L_n(f)-P_n}_2^2
&=
\norm{L_n(f-P_n)}_2^2\\
&=
\Big\lVert\sum_{j=0}^{n}(f(\zeta_{jn})-P_n(\zeta_{jn}))\ell_{jn}\Big\rVert_2^2\\
&=
\sum_{j=0}^{n}\norm{f(\zeta_{jn})-P_n(\zeta_{jn}))\ell_{jn}}_2^2\\
&\le 
M\norm{f-P_n}_\infty
\end{split}
\]
Therefore 
\[
\norm{f-L_n(f)}_2\le\norm{f-P_n}_2+\norm{L_n(f)-P_n}_2\underset{n\to\infty}{\longrightarrow}0,
\]
as we wanted to prove. 
\end{proof}

\bibliographystyle{acm}
\bibliography{BibTexOrtPol}

\end{document}